\documentclass[11pt]{article}

\usepackage[a4paper,margin=1in]{geometry}
\usepackage[T1]{fontenc}
\usepackage[utf8]{inputenc}
\usepackage{lmodern}
\usepackage{amsmath,amssymb,amsthm,mathtools}
\usepackage{microtype}
\usepackage{enumitem}
\usepackage{booktabs}
\usepackage[numbers,sort&compress]{natbib}
\usepackage[colorlinks=true,linkcolor=blue,citecolor=blue,urlcolor=blue]{hyperref}

\newtheorem{theorem}{Theorem}[section]
\newtheorem{lemma}[theorem]{Lemma}
\newtheorem{proposition}[theorem]{Proposition}

\newtheorem{conjecture}[theorem]{Conjecture}
\theoremstyle{definition}
\newtheorem{definition}[theorem]{Definition}
\theoremstyle{remark}

\newtheorem{question}[theorem]{Question}

\newcommand{\calF}{\mathcal F}
\newcommand{\calM}{\mathcal M}
\newcommand{\calP}{\mathcal P}
\newcommand{\calU}{\mathcal U}
\newcommand{\calS}{\mathcal S}

\newcommand{\inv}{\operatorname{inv}}

\title{Chain Covers in the Boolean Lattice}
\author{Zolt\'an Lóránt Nagy\thanks{
 E\"otv\"os Lor\'and University, Budapest, Hungary. The author is supported by  the J\'anos Bolyai Research Grant of the Hungarian Academy of Sciences and  by the NRDI EXCELLENCE grant, no.  151504;   	E-mail: {\tt zoltan.lorant.nagy@ttk.elte.hu.}}, \and Balázs Patkós\thanks{Alfr\'ed R\'enyi Institute of Mathematics, Budapest, Hungary and Department of Computer Science and Information Theory, Budapest University of
Technology and Economics, Budapest, Hungary, E-mail: {\tt patkos@renyi.hu}}}
\date{}

\begin{document}
\maketitle

\begin{abstract}
For integers $1\le r\le n+1$, let $N(n,r)$ denote the least number of chains in the Boolean lattice $B_n=2^{[n]}$ that cover every strict $r$-term chain.  The case $r=1$ is the classical chain-decomposition problem and is generalizing Dilworth's theorem and Sperner's theorem.  We study two complementary regimes.  First, when $r>1$ is fixed and $n\to\infty$.
Let   $$M(n,r):= \max_{\substack{
      a_0+\cdots+a_r=n\\
      a_0,a_r\ge 0,\ a_i\ge 1\ (1\le i\le r-1)
   }}
   \binom{n}{a_0,\ldots,a_r}.$$
We prove that lower and upper bounds which differ only by a logarithmic factor:
\[
  M(n,r)\le   N(n,r)\le \left(\frac r2+o(1)\right)\log n\cdot
 M(n,r).
\]

Second, we consider the near-maximal regime $N(n,n-t)$, where $t>0$ is fixed.  We prove 
\[
   N(n,n-t)= (\gamma_{t-1}+o(1))n!,
\]
where $\gamma_d$ is the limit of the density of a minimum size subset of the hypercube $Q_n$ that meets all $d$-dimensional subcube.  
\end{abstract}

\section{Introduction}

Let $P$ be a finite partially ordered set.  A \emph{chain} is a set of pairwise comparable elements, and an \emph{antichain} is a set of pairwise incomparable elements.  Dilworth's theorem states that the minimum number of chains in a decomposition of a finite poset equals the maximum size of an antichain \citep{Dilworth1950}.  In the Boolean lattice $B_n=2^{[n]}$ ordered by inclusion, Sperner's theorem identifies the largest antichains and gives the familiar value
\[
   \binom{n}{\lfloor n/2\rfloor}
\]
for the smallest chain decomposition of $B_n$ \citep{Sperner1928}.

This paper studies a higher-order covering version of the same problem.  Instead of covering elements of the Boolean lattice, we cover short chains by longer chains.

\begin{definition}
Let $C$ and $D$ be chains in $B_n$.  We say that $C$ \emph{covers} $D$ if $D\subseteq C$.  For $1\le r\le n+1$, let $N(n,r)$ be the least number of chains in $B_n$ that cover every strict chain of size $r$.
\end{definition}

Every chain in $B_n$ can be extended to a maximal chain.  Hence, throughout the paper, we may and shall take all covering chains to be maximal chains.  A maximal chain is identified with a permutation $\pi=(\pi_1,\ldots,\pi_n)$ of $[n]$, via
\[
   \varnothing
   \subset \{\pi_1\}
   \subset \{\pi_1,\pi_2\}
   \subset \cdots
   \subset [n].
\]
Thus the problem is equivalently a covering problem on the set of $r$-term chains, where the covering objects are permutations of $[n]$.


The basic obstruction comes from fixing levels of $B_n$.  A strict $r$-term chain
\[
   F_1\subset F_2\subset\cdots\subset F_r
\]
has a \emph{gap vector}
\[
 \textbf{ a}= (a_0,a_1,\ldots,a_r),
\]
where
\[
   a_0=|F_1|,\qquad
   a_i=|F_{i+1}\setminus F_i|\quad (1\le i\le r-1),\qquad
   a_r=n-|F_r|.
\]
The entries sum to $n$.  Conversely, fixing a gap vector fixes the ranks of the chain.  The number of chains with gap vector $(a_0,\ldots,a_r)$ is
\[
   \binom{n}{a_0,a_1,\ldots,a_r}.
\]

Let $\mathcal A_{n,r}$ be the set of admissible gap vectors, namely those with $a_0,a_r\ge0$, $a_i\ge1$ for $1\le i\le r-1$, and $a_0+\cdots+a_r=n$.  Define
\[
   M(n,r):=\max_{\mathbf a\in\mathcal A_{n,r}}\binom{n}{a_0,a_1,\ldots,a_r}.
\]
Except for the harmless endpoint effects when $r$ is very close to $n$, this is the balanced multinomial coefficient, where the $r+1$ parts are as equal as the admissibility constraints allow.  For fixed $r$ and large $n$, it is simply the multinomial coefficient with all $r+1$ parts differing by at most one.

\begin{proposition}[Layered lower bound]\label{prop:layered-lower}
For all $n$ and $r$,
\[
   N(n,r)\ge M(n,r).
\]
\end{proposition}

\begin{proof}
Choose an admissible gap vector $(a_0,\ldots,a_r)$ attaining $M(n,r)$, and consider all $r$-chains with these fixed gaps.  Their number is $M(n,r)$.  A maximal chain contains exactly one set of each rank, and hence covers exactly one chain from this fixed gap type.  Therefore at least $M(n,r)$ maximal chains are necessary.
\end{proof}

This lower bound is of the same spirit as several classical extremal results in the Boolean lattice.  Sperner's theorem is the rank-one case: a single middle layer is extremal.  Erd\H{o}s's extension of Sperner's theorem says that the largest family containing no chain of length $r+1$ is the union of the $r$ largest layers \citep{Erdos1945}.  Meshalkin's theorem gives a balanced multinomial bound for antichains of ordered set partitions \citep{Meshalkin1963}.  Katona's theorem for comparable pairs \citep{Katona1973}, the chain-profile-vector theorem of \citet{GerbnerPatkos2008}, and the later solution of Kleitman's chain-minimization conjecture by \citet{Samotij2019} are further examples where extremal or minimizing configurations are governed by layers as close to the middle as possible.

It is therefore natural to ask whether the layered lower bound is asymptotically sharp.

\begin{conjecture}\label{conj:fixed-r}
For every fixed $r\ge 1$,
\[
   N(n,r)=(1+o(1))M(n,r).
\]
\end{conjecture}

The exact equality version is false in general; see Section~\ref{sec:small-values}.  The asymptotic question remains open already for $r=2$. 

Dilworth's theorem has several generalisation. A particularly important higher-order extension is the theorem of Greene and Kleitman \cite{Greene}, which refines Dilworth's theorem by considering the maximum size of a union of $k$ antichains and the corresponding saturated chain partitions.  Our problem is different in nature: instead of covering the elements of a poset, or unions of antichains, we ask for chains which cover all smaller chains of a prescribed length.

\subsection*{Main results}

We prove results in two regimes.

First, for fixed $r$, the layered lower bound is sharp up to a logarithmic factor. It is worth noting that the total number $T(n,r)$ of $r$-chains in $B_n$ satisfies $(r+1)^n-(r+1)r^n\le T(n,r)\le (r+1)^n$. So by Stirling's formula we have $T(n,r)\le M(n,r)\cdot n^{\frac{r+1}{2}}$. Clearly, we have $N(n,r)\le T(n,r)$ by covering the $r$-chains one-by-one. The upper bound of the next theorem could be compared to this trivial upper bound.

\begin{theorem}[Fixed chain length]\label{thm:fixed-r}
For every fixed $r\ge 1$,
\[
   N(n,r)\le
   \left(\frac r2+o(1)\right)\log n\,M(n,r).
\]
\end{theorem}
The proof appears in Section~\ref{sec:fixed}. It is a deterministic fractional-covering argument. One gives all maximal chains equal weight so that the total weight is exactly $M(n,r)$, then rounds the fractional cover after separating gap types according to their distance from the balanced type. The same rounding method also has a uniform form for arbitrary $r=r(n)$: without assuming $r$ fixed, one obtains
\[
N(n,r)\le M(n,r)H_{\binom{n+1}{r}}
\le M(n,r)\left(1+\log\binom{n+1}{r}\right),
\]
where $H_q=1+1/2+\cdots+1/q$ A slightly sharper version, stated in Section~\ref{sec:fixed}, rounds only the dense gap types and covers the sparse remainder individually.


Second, when the chain length is close to maximal, the layered lower bound is not always the correct scale.  Write the chain length as $n-t$, where $t$ is fixed.  For $n$ large enough in terms of $t$, the layered lower bound is exactly
\[
   M(n,n-t)=\frac{n!}{2^{t-1}},
\]
corresponding to the balanced type with $t-1$ gaps of size $2$ and the rest of size $1$.  We connect the problem of determining $N(n,n-t)$ to the problem of finidng the smallest subset of the hypercube hitting all $d$-dimensional subcubes with $d=t-1$.





\begin{theorem}\label{thm:near-maximal-gamma}
Fix $d\ge 1$. Let $g(m,d)$ denote the minimum size of a subset of $Q_m=\{0,1\}^m$ which intersects every $d$-dimensional subcube of $Q_m$, and suppose that
\[
\gamma_d=\lim_{m\to\infty}\frac{g(m,d)}{2^m}
\]
exists. Then
\[
N(n,n-d-1)=(\gamma_d+o(1))n!.
\]
Equivalently, writing $t=d+1$,
\[
N(n,n-t)=(\gamma_{t-1}+o(1))n!
\]
for every fixed $t\ge 2$.
\end{theorem}

The constants $\gamma_d$ are classical objects in hypercube extremal theory.  It is known that $\gamma_1=1/2$ and $\gamma_2=1/3$ \citep{Kostochka1976,JohnsonEntringer1989}.  


We also determine $N(n,n-2)$ exactly.
\begin{proposition}\label{cor:nminus2}
For every $n\ge 3$,
\[
   N(n,n-2)=\frac{n!}{2}.
\]
\end{proposition}

The paper is organized as follows. We give an upper bound on $N(n,r)$ in the case when $r\ll n$ and prove Theorem \ref{thm:fixed-r} in
Section~\ref{sec:fixed}, using fractional covers. We also discuss the case $r=2$ in detail. Then we  discuss the case of $N(n,n-t)$ when $t$ is small and prove Theorem \ref{thm:near-maximal-gamma} and Proposition \ref{cor:nminus2} in Section~\ref{sec:nearmax} by making connection to the vertex Turán number of subcubes in a large dimensional hypercube and by partitioning the permutations by their inverson numbers. Finally, we pose an open problem in Section \ref{sec:open}.

\section{Fixed chain length}\label{sec:fixed}

This section proves Theorem~\ref{thm:fixed-r}.  We begin with the set-cover rounding theorem used in the proof.

\subsection{Lov\'asz--Stein fractional cover rounding}

We use the following standard form of the Lov\'asz--Stein theorem \citep{Johnson1974,Stein1974,Lovasz1975}.

\begin{theorem}[Lov\'asz--Stein]\label{thm:LS}
Let $\calU$ be a finite universe and let $\calS\subseteq 2^{\calU}$ be a family of subsets covering $\calU$.  Suppose that $(w_S)_{S\in\calS}$ is a fractional cover, that is,
\[
   w_S\ge 0,
   \qquad
   \sum_{S\ni u}w_S\ge 1
   \quad\text{for every }u\in\calU.
\]
If $\tau^*=\sum_{S\in\calS}w_S$ and each $S\in\calS$ has size at most $q$, then $\calU$ has an integral cover using at most
\[
   H_q\tau^*
\]
sets from $\calS$, where $H_q=1+1/2+\cdots+1/q\le 1+\log q$.
\end{theorem}

\subsection{Gap types and dense types}

Fix $r\in \mathbb{Z}^+$.  For a gap vector
\[
   \mathbf a=(a_0,a_1,\ldots,a_r),
   \qquad a_0+\cdots+a_r=n,
\]
let $\calP_{\mathbf a}$ be the set of strict $r$-chains with this gap vector and write
\[
   T_{\mathbf a}:=|\calP_{\mathbf a}|=\binom{n}{a_0,a_1,\ldots,a_r}.
\]
Let $\mathbf m=(m_0,
\ldots,m_r)$ be a balanced vector, so that
\[
   M(n,r)=\binom{n}{m_0,
\ldots,m_r}.
\]
Define
\[
   \rho(\mathbf a):=\frac{M(n,r)}{T_{\mathbf a}}.
\]
We call $\mathbf a$ \emph{dense} if
\[
   \rho(\mathbf a)\le n^{r+2},
\]
and \emph{sparse} otherwise.

The following estimate says that there are few dense types.

\begin{lemma}\label{lem:dense-general}
For fixed $r$, the number of dense gap types is
\[
   O_r\bigl((n\log n)^{r/2}\bigr).
\]
\end{lemma}

\begin{proof}
By Stirling's formula,
\[
   \log\binom{n}{a_0,
\ldots,a_r}=nH(p_0,
\ldots,p_r)+O_r(\log n),
\]
where $p_i=a_i/n$ and $H(\textbf{p})=-\sum_i p_i\log p_i$. Similarly, since \(\mathbf m\) is balanced,
$\log M(n,r)=n\log(r+1)+O_r(\log n).$
Let
\[
   \textbf{u}=\left(\frac1{r+1},\ldots,\frac1{r+1}\right).
\]
Then
\[
   \log(r+1)-H(\textbf{p})
   =
   D(\textbf{p}\|\textbf{u}),
\]
where \(D(\cdot\|\cdot)\) denotes the relative entropy (or Kullback-Leibler divergence). Pinsker's inequality claims  that \[
   D(\textbf{p}\|\textbf{u})
   \ge \frac12\|\textbf{p}-\textbf{u}\|^2.
\]
Consequently
\[
   \log(r+1)-H(p_0,\ldots,p_r)
   \ge
   \frac12\sum_{i=0}^r
   \left(p_i-\frac1{r+1}\right)^2 .
\]

Therefore
\[
   \log\frac{M(n,r)}{T_{\mathbf a}}
   \ge
   \frac n2\sum_{i=0}^r
   \left(p_i-\frac1{r+1}\right)^2
   -O_r(\log n).
\]

If $\mathbf a$ is dense, then
\[
   \log\frac{M(n,r)}{T_{\mathbf a}}
   \le (r+2)\log n.
\]
Combining this with the entropy estimate gives
\[
   \sum_{i=0}^r\left(a_i-\frac{n}{r+1}\right)^2
   =O_r(n\log n).
\]
The integer solutions to $a_0+\cdots+a_r=n$ in this ball form an $r$-dimensional lattice set of size $O_r((n\log n)^{r/2})$.
\end{proof}

\subsection{The fractional cover}

Let $\calM$ be the set of all maximal chains of $B_n$.  Thus $|\calM|=n!$.  Assign every maximal chain the weight
\[
   w_C:=\frac1{m_0!m_1!\cdots m_r!}.
\]
The total weight is
\[
   \sum_{C\in\calM}w_C
   =\frac{n!}{m_0!m_1!\cdots m_r!}
   =M(n,r).
\]

We claim that this is a fractional cover of every strict $r$-chain.  Indeed, let $F_1\subset\cdots\subset F_r$ have gap vector $\mathbf a=(a_0,\ldots,a_r)$.  A maximal chain contains it precisely when its corresponding permutation orders the elements inside each gap in the prescribed block order.  Hence the number of maximal chains extending it is
\[
   a_0!a_1!\cdots a_r!.
\]
The total weight of the maximal chains containing this $r$-chain is therefore
\[
   \frac{a_0!a_1!\cdots a_r!}{m_0!m_1!\cdots m_r!}.
\]
Since $\mathbf m$ maximizes the multinomial coefficient among admissible gap vectors,
\[
   \binom{n}{a_0,
\ldots,a_r}
   \le
   \binom{n}{m_0,
\ldots,m_r},
\]
which is equivalent to
\[
   a_0!a_1!\cdots a_r!
   \ge
   m_0!m_1!\cdots m_r!.
\]
Thus every $r$-chain receives fractional coverage at least one.

\subsection{Rounding dense types and discarding sparse types}

Let $\calU_D$ be the set of all strict $r$-chains of dense gap type.  The fractional cover above has total weight $M(n,r)$ on $\calU_D$.

A maximal chain contains exactly  one $r$-chain of each fixed gap type: the gap vector fixes the ranks of the $r$ members, and a maximal chain has exactly one set of each rank.  Hence a maximal chain contains exactly
\[
   s_n=O_r((n\log n)^{r/2})
\]
dense-type $r$-chains.  By Theorem~\ref{thm:LS}, the dense part can be covered using at most
\[
   H_{s_n}M(n,r)
\]
maximal chains.  Since
\[
   H_{s_n}=\log s_n+O(1)
   =\left(\frac r2+o(1)\right)\log n,
\]
the dense types are covered by at most
\[
   \left(\frac r2+o(1)\right)\log n\,M(n,r)
\]
maximal chains.

It remains to cover sparse types.  There are $O_r(n^r)$ possible gap types.  If $\mathbf a$ is sparse, then
\[
   T_{\mathbf a}<\frac{M(n,r)}{n^{r+2}}.
\]
Thus the total number of sparse $r$-chains is at most
\[
   O_r(n^r)\frac{M(n,r)}{n^{r+2}}
   =O_r\left(\frac{M(n,r)}{n^2}\right)
   =o(M(n,r)).
\]
Cover these sparse chains one by one by arbitrary maximal extensions.  This costs only $o(M(n,r))$ additional maximal chains, which is negligible relative to the logarithmic main term.  This proves Theorem~\ref{thm:fixed-r}.

\subsection{A uniform version of the rounding argument}\label{subsec:uniform-rounding}

The preceding proof was optimized for fixed \(r\), where the number of dense gap types can be estimated by a Gaussian bound around the balanced multinomial coefficient. The underlying fractional-covering argument, however, does not require \(r\) to be fixed. For arbitrary \(1\le r\le n+1\), a maximal chain contains exactly  one \(r\)-chain of each rank pattern, and there are exactly \(\binom{n+1}{r}\) possible rank patterns. Applying Theorem~\ref{thm:LS} directly to the fractional cover from the previous subsection therefore gives the uniform estimate
\[
N(n,r)\le M(n,r)H_{\binom{n+1}{r}}
\le M(n,r)\left(1+\log\binom{n+1}{r}\right).
\]
Equivalently,
\[
N(n,r)\le M(n,r)\left(1+r\log\frac{e(n+1)}{r}\right),
\]
which is meaningful for every function \(r=r(n)<n\).

One may also keep the dense--sparse refinement in this variable-\(r\) setting. Given \(A>1\), let \(D_{n,r}(A)\) be the number of gap types \(\mathbf a\) for which \(M(n,r)/T_{\mathbf a}\le A\). Rounding only these dense types and covering all remaining chains one by one gives
\[
N(n,r)
\le M(n,r)H_{D_{n,r}(A)}
+ \frac{\binom{n+1}{r}}{A}M(n,r).
\]
For fixed \(r\), choosing \(A=n^{r+2}\) and estimating \(D_{n,r}(A)\) by Lemma~\ref{lem:dense-general} recovers Theorem~\ref{thm:fixed-r}. For growing \(r\), this form is often the more appropriate statement, since the quality of the bound is governed by the number of near-balanced gap types rather than by a fixed-dimensional Gaussian estimate.

\subsection{Small values cases}\label{sec:small-values}

The layered lower bound is the right first obstruction, but it is not an exact formula for all finite $n$ and $r$.  The first exception for two-term chains is
\[
   M(5,2)=\binom{5}{1,2,2}=30,
   \qquad
   N(5,2)=31.
\]
For $r=2$, exact computations via exhaustive search give
\[
\begin{array}{c|ccccccccc}
 n        &1&2&3&4&5&6&7&8&9\\
\hline
 M(n,2)   &1&2&6&12&30&90&210&560&1680\\
 N(n,2)   &1&2&6&12&31&90&210&560&1680.
\end{array}
\]
Thus equality fails at $n=5$ but holds again for $6\le n\le 9$.

A second failure occurs in the near-maximal diagonal:
\[
   M(6,3)=\binom{6}{1,1,2,2}=180,
   \qquad
   N(6,3)>180.
\]
These computations are not used in the proof of Theorem~\ref{thm:fixed-r}.  They show, however, that the lower bound of Proposition \ref{prop:layered-lower} on $N(n,k)$ can not be attained for all pairs $n,k$, thus the error term in Conjecture~\ref{conj:fixed-r} is necessary in general. 

\section{Near-maximal chains}\label{sec:nearmax}

We now study $N(n,n-t)$ for fixed $t$.  The chain length is close to maximal, and the permutation model becomes especially useful. The main goal of this section is to prove Theorem \ref{thm:near-maximal-gamma}, but first we prove a weaker bound that gives an exact upper bound for the case $t=2$.

\subsection{An inversion-number construction}

For a permutation $\pi\in S_n$, let $\inv(\pi)$ denote its inversion number.  

\begin{theorem}\label{thm:mod-upper-proof}
For every $t\ge 1$ and every $n> t$,
\[
   N(n,n-t)\le \frac{n!}{t}.
\]
\end{theorem}

\begin{proof}
For \(s\in\{0,\ldots,t-1\}\), let
\[
   \mathcal F_s
   =
   \{\pi\in S_n:\operatorname{inv}(\pi)\equiv s\pmod t\}.
\]
Choose $s$ such that $|F_s|\le \frac{n!}{t}$. Such $s$ exists by averaging. 
Select all maximal chains corresponding to permutations $\pi\in S_n$ with
\[
   \inv(\pi)\equiv s\pmod t,
\]


It remains to show that the selected maximal chains cover every $(n-t)$-term chain.  Let
\[
  F_1\subset\cdots\subset F_{n-t}
\]
be such a chain, with gap sizes
\[
   a_0,a_1,\ldots,a_{n-t}.
\]
Thus there are $n-t+1$ gaps and their sizes sum to $n$.  A maximal extension is obtained by ordering the elements inside each gap.  The inversions between different gaps are fixed; only the internal inversion numbers can vary.

For a gap of size $a$, the possible internal inversion numbers are all integers from $0$ to $\binom a2$.  Therefore the possible total internal inversion numbers contain every integer from $0$ to
\[
   L:=\sum_i\binom{a_i}{2}.
\]
We have
\[
   L\ge \sum_{a_i>0}(a_i-1)=n-h,
\]
where $h$ is the number of nonempty gaps.  Since $h\le n-t+1$, it follows that
\[
   L\ge t-1.
\]
Thus the maximal extensions of the given chain realize $t$ consecutive inversion-number values modulo $t$, and hence all residue classes modulo $t$.  In particular, one extension has inversion number $s\pmod t$.
\end{proof}

\begin{proof}[Proof of Proposition \ref{cor:nminus2}]
As the layered lower bound on $N(n,n-t)$ is $M(n,n-t)=\frac{n!}{2^{t-1}}$.  As a consequence,  by Theorem \ref{thm:mod-upper-proof} we obtain $N(n,n-2)=\frac{n!}{2}$.      
\end{proof}



\subsection{Subcube hitting and improved lower bounds}


Recall that for $m,d\ge 0$, let $g(m,d)$ be the minimum size of a set $H\subseteq\{0,1\}^m$ such that every $d$-dimensional subcube of $Q_m$ contains at least one point of $H$.  Define
\[
   \gamma_d(m):=\frac{g(m,d)}{2^m},
   \qquad
   \gamma_d:=\lim_{m\to\infty}\gamma_d(m).
\]
The limit exists by a standard averaging argument.  Equivalently, $1-\gamma_d$ is the vertex Tur\'an density for avoiding $Q_d$ in the hypercube.

These constants have a substantial literature.  The case $d=2$ was determined by \citet{Kostochka1976} and independently by \citet{JohnsonEntringer1989}, giving $\gamma_2=1/3$.  Vertex Tur\'an problems for hypercubes were developed further by \citet{JohnsonTalbot2010}.  The related covering and polychromatic questions were raised and studied by \citet{AlonKrechSzabo2007}, \citet{Offner2008}, and \citet{OzkahyaStanton2015}.  For general $d$, Alon, Krech, and Szabó \cite{AlonKrechSzabo2007} gave the bounds
\[
   \frac{\log_2(d+2)}{2^{d+2}}
   \le \gamma_d\le \frac1{d+1}.
\]
The upper bound is the simple layer-modulo construction in the hypercube.  It was long natural to ask whether $\gamma_d=1/(d+1)$ always holds.  This fails for large $d$: Ellis, Ivan and Leader \cite{EllisIvanLeader2024} proved $\gamma_d<1/(d+1)$ for all $d\ge 8$, and in fact $\gamma_d\le c^d$ for an absolute constant $c<1$; Ellis, Ivan and Leader \cite{EllisIvanLeader2025} extended the strict inequality to $d=6,7$.  The exact values of $\gamma_d$ remain open for $d\ge 3$.

We now show how these constants enter our chain-cover problem.


\begin{proof}[Proof of Theorem \ref{thm:near-maximal-gamma}]
Let $\calF\subseteq S_n$ be any family of permutations whose maximal chains cover every $(n-t)$-term chain.  We partition $S_n$ into copies of $Q_m$ as follows.  Group the positions into adjacent pairs
\[
   (1,2),(3,4),\ldots,(2m-1,2m),
\]
with one leftover position if $n$ is odd.  Two permutations are in the same orbit if they differ only by independently swapping the two entries inside these adjacent pairs.  Each orbit has size $2^m$ and is naturally a copy of $Q_m$.

Fix one orbit $Q\cong Q_m$.  We claim that $\calF\cap Q$ meets every $d$-dimensional subcube of $Q$.  Indeed, choose any set $I\subseteq[m]$ of size $d=t-1$.  The corresponding $d$-subcube consists of varying precisely the swaps in the adjacent pairs indexed by $I$, while all other orientations are fixed.  Merge those $d$ adjacent pairs into unordered doubleton gaps, and leave all other positions as singleton gaps, in the order prescribed by the orbit.  This produces an ordered gap partition of $[n]$ with
\[
   n-d=n-t+1
\]
gaps, hence an $(n-t)$-term chain.  Its maximal extensions inside the orbit are exactly the vertices of the chosen $d$-subcube.  Since $\calF$ covers every $(n-t)$-term chain, $\calF\cap Q$ must meet this subcube.

Thus, in each orbit,
\[
   |\calF\cap Q|\ge g(m,d).
\]
There are $n!/2^m$ orbits.  Summing over all orbits gives
\begin{equation}
       |\calF|
   \ge
   \frac{n!}{2^m}g(m,d)
   =\gamma_d(m)n!.
\end{equation}

Now let us prove a matching upper bound.
Let $m=n-1$, and let $H\subseteq Q_m$ be a set meeting every
$d$-dimensional subcube of $Q_m$, with
\[
|H|=(\gamma_d+o(1))2^m.
\]
We shall construct a family of maximal chains in $B_n$ from a suitable
coordinate permutation of $H$.

For a permutation $\pi\in S_n$, define its inverse Lehmer code by
\[
L_x(\pi)=
\bigl|\{y<x:\ y\text{ appears after }x\text{ in }\pi\}\bigr|,
\qquad 1\le x\le n.
\]   

Note that the  inverse Lehmer code of a permutation is a standard inversion-table representation of permutations; see, see Knuth~\cite{Knuth11}.
 Applying this, we have
\[
0\le L_x(\pi)\le x-1,
\]
and the map
\[
\pi\mapsto (L_1(\pi),\ldots,L_n(\pi))
\]
is a bijection from $S_n$ to
\[
\{0\}\times\{0,1\}\times\cdots\times\{0,1,\ldots,n-1\}.
\]
Define the parity vector
\[
P(\pi)=
\bigl(L_2(\pi),L_3(\pi),\ldots,L_n(\pi)\bigr)\pmod 2
\in Q_{n-1}.
\]

We first prove the covering property. Let
\[
C:\quad A_1\subset A_2\subset\cdots\subset A_{n-d-1}
\]
be an arbitrary $(n-d-1)$-chain. It determines an ordered gap
partition
\[
B_0,B_1,\ldots,B_{n-d-1}
\]
of $[n]$, where
\[
B_0=A_1,\qquad
B_i=A_{i+1}\setminus A_i\quad (1\le i\le n-d-2),
\qquad
B_{n-d-1}=[n]\setminus A_{n-d-1}.
\]
The maximal chains extending $C$ are obtained by choosing, independently
for each gap $B_i$, an ordering of the elements of $B_i$, while the
order of the gaps themselves is fixed.

Consider one nonempty gap
\[
B={b_1<b_2<\cdots<b_s}.
\]
The order of the elements of $B$ is free. The contributions to
$L_{b_j}(\pi)$ coming from elements outside $B$ are fixed over all
maximal extensions of $C$. The variable part is the internal inverse
Lehmer coordinate of $b_j$ inside the ordering of $B$, which ranges
independently over $0,1,\ldots,j-1$.
Consequently, the parities of
\[
L_{b_2}(\pi),L_{b_3}(\pi),\ldots,L_{b_s}(\pi)
\]
realize all $2^{s-1}$ binary patterns, up to a fixed translation in
$\mathbb F_2^{s-1}$. Applying this independently to all gaps, the set
of parity vectors $P(\pi)$, as $\pi$ ranges over all maximal
extensions of $C$, contains a Boolean subcube of dimension
\[
\sum_i (|B_i|-1),
\]
where the sum is over the nonempty gaps. Since the number of gaps is
$n-d$ and their total size is $n$, this dimension is at least $d$.
Thus the $P$-image of the maximal extensions of $C$ contains a
$d$-dimensional subcube of $Q_{n-1}$.

Every coordinate permutation of $H$ still meets every $d$-dimensional
subcube. Therefore, if $H'$ is any coordinate-permuted copy of $H$,
then every $(n-d-1)$-chain $C$ has a maximal extension $\pi$ with
$P(\pi)\in H'$.
Hence the family
$\mathcal F(H')
=
\{\pi\in S_n:\ P(\pi)\in H'\}$
covers all $(n-d-1)$-chains.

It remains to choose the coordinate permutation so that $\mathcal F(H')$
has the desired size. Let $\pi$ be uniformly random in $S_n$. The
inverse Lehmer coordinates $L_x(\pi)$ are independent, and $L_x(\pi)$
is uniformly distributed on $\{0,1,\ldots,x-1\}$. Hence the coordinates
of $P(\pi)$ are independent Bernoulli random variables. If $x$ is
even, then
\[
\mathbb P(L_x(\pi)\equiv 1\pmod 2)=\frac12,
\]
while if $x$ is odd, then
\[
\mathbb P(L_x(\pi)\equiv 1\pmod 2)=
\frac12-\frac1{2x}.
\]
Let $W=|P(\pi)|$. Then $W$ is a sum of independent Bernoulli random
variables whose parameters differ from $1/2$ by $O(1/x)$. Therefore
\[
d_{\mathrm{TV}}\bigl(W,\operatorname{Bin}(n-1,1/2)\bigr)=o(1),
\] where $d_{\mathrm{TV}}$ denotes the total variance distance for a pair of probability distributions \(\mu,\nu\) on a finite set, i.e., 
\[
   d_{\mathrm{TV}}(\mu,\nu)
   =
   \frac12\sum_x |\mu(x)-\nu(x)|.
\]
This bound of the distance follows from the standard local central limit theorem
for triangular arrays: the mean of $W$ differs from $(n-1)/2$ by only
$O(\log n)=o(\sqrt n)$, while $\operatorname{Var}(W)=\frac{n-1}{4}+O(1)$.

Now choose a uniformly random coordinate permutation $\sigma$ of
$Q_{n-1}$. Write
\[
H_\sigma=\sigma(H).
\]
For each $j$, let
\[
\alpha_j=
\frac{|H\cap \{x\in Q_{n-1}: |x|=j\}|}{\binom{n-1}{j}}
\]
be the density of $H$ on the $j$-th Hamming layer. Averaging over
coordinate permutations gives
\[
\mathbb E_\sigma(
\mathbb P(P(\pi)\in H_\sigma)=
\sum_{j=0}^{n-1}\mathbb P(W=j)\alpha_j.
\]
On the other hand,
\[
\frac{|H|}{2^{n-1}}=
\sum_{j=0}^{n-1}
\mathbb P\bigl(\operatorname{Bin}(n-1,1/2)=j\bigr)\alpha_j.
\]
Since $0\le \alpha_j\le 1$ for all $j$, and since the two weight
distributions have total variation distance $o(1)$, we get
\[
\mathbb E_\sigma(
\mathbb P(P(\pi)\in H_\sigma)
\le
\frac{|H|}{2^{n-1}}+o(1)=
\gamma_d+o(1).
\]
Hence there exists a coordinate permutation $\sigma$ such that, with
$H'=H_\sigma$,
\[
\mathbb P(P(\pi)\in H')\le \gamma_d+o(1).
\]
For this choice,
\[
|\mathcal F(H')|\le 
n!\mathbb P(P(\pi)\in H')
\le
(\gamma_d+o(1))n!.
\]
As shown above, $\mathcal F(H')$ covers all $(n-d-1)$-chains. Thus
\[
N(n,n-d-1)\le(\gamma_d+o(1))n!.
\]
Together with the lower bound, this proves
\[
N(n,n-d-1)=(\gamma_d+o(1))n!.
\]
\end{proof}




\section{Concluding remarks and open questions}\label{sec:open}


For fixed chain length $r$, the main open problem is whether the logarithmic factor in Theorem~\ref{thm:fixed-r} can be removed.

\begin{question}
For every fixed $r\ge 2$, is it true that
\[
   N(n,r)=(1+o(1))M(n,r)?
\]
In particular, is
\[
   N(n,2)=(1+o(1))\binom{n}{m_0,m_1,m_2}
\]
for a balanced triple $(m_0,m_1,m_2)$?
\end{question}

The proof of Theorem~\ref{thm:fixed-r} does not attempt to coordinate choices between balanced intervals; it rounds a fractional cover.  Removing the logarithm would require a much finer construction, closer to a global matching problem.




\medskip

\noindent \textbf{AI declaration.} The authors found all the main ideas of the proof of Theorem \ref{thm:fixed-r}, the lower bound proof of Theorem \ref{thm:near-maximal-gamma} and the proof of Theorem \ref{thm:mod-upper-proof}, they used ChatGPT during drafting and checking of parts of the exposition. After the first version was uploaded to arxiv, M\'at\'e Vizer prompted ChatGPT about possible improvement and it returned the upper bound proof of Theorem \ref{thm:near-maximal-gamma}. We thank M\'at\'e for sharing the result. The authors  take full responsibility for the mathematical content.

\end{document}